\documentclass[12pt]{article}
\usepackage{amsfonts}
\usepackage{mathrsfs,dsfont,bbm}

\usepackage[latin1]{inputenc}
\usepackage{amsmath,amssymb}
\usepackage{latexsym}
\usepackage{epstopdf}
\usepackage{geometry}
\usepackage[active]{srcltx}
\usepackage[
bookmarks=true,         
bookmarksnumbered=true, 
colorlinks=true, pdfstartview=FitV, linkcolor=blue, citecolor=blue,
urlcolor=blue]{hyperref}

\newtheorem{theorem}{Theorem}[section]

\newtheorem{lemma}[theorem]{Lemma}
\newtheorem{proposition}[theorem]{Proposition}

\newtheorem{remark}[theorem]{Remark}

\numberwithin{equation}{section}
\newenvironment{proof}[1][Proof]{\noindent\textit{#1.} }{\hfill \rule{0.5em}{0.5em}}

\begin{document}

\title{Limit laws for functionals of self-intersection symmetric $\alpha$-stable processes}
\date{\today}

\author
{Minhao Hong, Qian Yu \footnote{Qian Yu is supported by  the Fundamental Research Funds for the Central Universities (NS2022072), National Natural Science Foundation of China (12201294) and Natural Science Foundation of Jiangsu Province, China (BK20220865).}}

\maketitle

\begin{abstract}
In this paper, we prove two limit laws for functionals of self-intersection symmetric $\alpha$-stable processes with $\alpha\in(1,2)$. The
results are obtained based on the method of moments, the sample configuration and the chaining argument introduced in \cite{NuXu} are employed.
\end{abstract}

\section{Introduction}\label{sec1}
Let $X=\{X_t, t\geq0\}$ be a symmetric $\alpha$-stable process in $\mathbb{R}$. If $\alpha>1$, then the local time of $X$ exists (see in \cite{Boy64}) and can be defined as
\begin{equation}\label{sec1-eq1.0}
l_t(x)=\int_0^t\delta(X_s-x)ds, ~~t\geq0, x\in\mathbb{R},
\end{equation}
where $\delta$ is the Dirac delta function.
For any integrable function $f:~ \mathbb{R}\to\mathbb{R}$, using the scaling property of $\alpha$-stable process and the continuity of the local
time, we can easily obtain the convergence in law in the space $C([0,\infty))$ as follows
\begin{equation}\label{sec1-eq1.1}
n^{\frac{1-\alpha}{\alpha}}\int_0^{nt}f(X_s)ds\overset{law}{\to}l_t(0)\int_{\mathbb{R}}f(x)dx, ~~\text{as}~ n\to\infty.
\end{equation}

If we add a condition $\int_{\mathbb{R}}f(x)dx=0$, we can see $n^{\frac{1-\alpha}{\alpha}}\int_0^{nt}f(X_s)ds$ converges to zero, as $n\to\infty$.
But multiply the left side of \eqref{sec1-eq1.1} by a factor $n^\beta$ ($\beta>0$), then the right side of \eqref{sec1-eq1.1} converges to a nonzero process.
This has been proved to be true, Rosen \cite{Ros91} showed
$$
n^{\frac{1-\alpha}{2\alpha}}\int_0^{nt}f(X_s)ds\overset{law}{\to}c_f W(l_t(0)), ~~\text{as}~ n\to\infty,
$$
where $W$ is a real-valued Brownian motion independent of $\alpha$-stable process $X$, $c_f$ is a constant dependent on $f$ and $\alpha$.

If we replace $X_s$ with $X_s-X_r$ ($0<r<s<t$) in \eqref{sec1-eq1.0}, we can define self-intersection local time of $X$ as
\begin{equation}\label{sec1-eq1.2}
L_t(x)=\int_0^t\int_0^s\delta(X_s-X_r-x)drds, ~~t\geq0, x\in\mathbb{R}.
\end{equation}
The self-intersection local time $L_t(0)$ measures the amount of time that the process $X$ spends intersecting itself on the time interval $[0,t]$. Under condition $\alpha\in(1,2)$, we give its existence in Proposition \ref{sec2-pro2.1}.

However, to the best of our knowledge, the limit laws for functionals of self-intersection symmetric
$\alpha$-stable processes has not been considered in the literature. Motivated by the aforementioned works, we will consider this problem in
this paper and the main results are as follows.
\begin{theorem}\label{Thm-first-order}
Suppose that $f$ is bounded and $\int_{\mathbb{R}}|xf(x)|<\infty$. Then, for any $t>0$,
$$\frac{1}{n^{2-\frac1\alpha}}\int_0^{nt}\int_0^uf(X_u-X_v)dvdu\overset{law}{\to}\left(\int_{\mathbb{R}}f(x)dx\right)L_t(0),$$
as $n\to\infty$, where $X$ is a symmetric $\alpha$-stable process with parameter $\alpha\in(1,2)$ and $L_t(0)$ is the self-intersection local time of $X$ at $0$.
\end{theorem}

\begin{theorem}\label{Thm-second-order}
Under the assumptions in Theorem \ref{Thm-first-order}, we further assume that  $\int_{\mathbb{R}}f(x)dx=0$. Then, for any $t>0$,
$$\frac{1}{n^{\frac{3\alpha-1}{2\alpha}}}\int_0^{nt}\int_0^uf(X_u-X_v)dvdu\overset{law}{\to}\left(\frac1{4\pi^2}\int_{\mathbb{R}}|\widehat{f}(x)|^2|x|^{-\alpha}dx\right)^{\frac12}Z(t),$$
as $n\to\infty$, where $\widehat{f}$ is the Fourier transform of $f$ and $Z(t)$ is a random variable with parameter $t>0$ and $\mathbb{E}(Z(t))^m=\frac{m!\Gamma^{\frac{m}{2}}(1-\frac1\alpha)\left(\int_{\mathbb{R}}e^{-|x|^\alpha}dx\right)^{\frac{m}{2}}}{\Gamma(\frac{3m}{2}+1-\frac{m}{2\alpha})}t^{\frac{3m}{2}-\frac{m}{2\alpha}}$ if $m$ is even, $\mathbb{E}(Z(t))^m=0$ if $m$ is odd.
\end{theorem}

\begin{remark}\label{sec1-remark}
The assumption $\int_{\mathbb{R}}|xf(x)|dx<\infty$ implies that $f\in L^p(\mathbb{R})$ for any $p\geq1$, and
\begin{equation}\label{sec1-eq1.3}
\int_{|x|>1}|\widehat{f}(x)|^2|x|^{-\gamma}dx<\infty
\end{equation}
for any $\gamma>1/2$. In fact, the function $f$ is bounded and
\begin{align*}
\int_{\mathbb{R}}|f(x)|^pdx&\leq \int_{|x|\leq1}|f(x)|^pdx+\int_{|x|>1}|f(x)|^p|x| dx\\
&\leq |f(x)|^p+|f(x)|^{p-1}\int_{\mathbb{R}}|f(x)||x| dx<\infty.
\end{align*}
So,
\begin{align*}
\int_{|x|>1}|\widehat{f}(x)|^2|x|^{-\gamma}dx&\leq \left(\int_{|x|>1}|\widehat{f}(x)|^4dx\right)^{1/2}\left(\int_{|x|>1}|x|^{-2\gamma}dx\right)^{1/2}\\
&\leq c\left(\int_{|x|>1}|f(x)|^4dx\right)^{1/2}<\infty,
\end{align*}
where we use the Plancherel theorem and $f\in L^{4}(\mathbb{R})$.

Moreover, the additional condition $\int_{\mathbb{R}}f(x)dx=0$ gives
$$\left|\widehat{f}(x)-\int_{\mathbb{R}}f(x)dx\right|=|\widehat{f}(x)-\widehat{f}(0)|\leq c\left(|x|^\beta\wedge 1\right),  ~~\beta\in[0,1].$$
This means 
\begin{align}\label{sec1-eq1.3-1}
\int_{|x|\leq1}|\widehat{f}(x)|^2|x|^{-\alpha}dx
&\leq c\int_{|x|\leq1}|x|^{2\beta-\alpha}dx<\infty.
\end{align}
Together \eqref{sec1-eq1.3} and \eqref{sec1-eq1.3-1}, we have
$\int_{\mathbb{R}}|\widehat{f}(x)|^2|x|^{-\alpha}dx<\infty$.

\end{remark}

The study of self-intersection local time has attracted the attention of many scholars.
Hu \cite{Hu1996} discussed the exact smoothness of the self-intersection local time of Brownian motion in the sense of Meyer-Watanabe.
If the Brownian motion is replaced by a more general Gaussian process (fractional Brownian motion),
Hu and Nualart \cite{Hu2005} proved existence condition of the renormalized self-intersection local time for fractional Brownian motion, and given two central
limit theorems for nonexistence conditions. Other studies on self-intersection local time for Gaussian process can be referred to \cite{Hu1996}, \cite{Jara17}, \cite{Jung2014},  \cite{Jung2015},  \cite{Mar08}, \cite{Yan2015}, \cite{Yu21} and the references therein.

Moreover, the self-intersection local time for non-Gaussian case is also concerned by many scholars, especially the famous symmetric $\alpha$-stable processes. Rosen \cite{Ros05} showed the renormalized self-intersection local time is differentiable in the spatial variable, Nualart and Xu \cite{NuXu} proved two limit laws for functionals of one dimensional symmetric $\alpha$-stable process with $\alpha=1$.  Other studies on $\alpha$-stable process can be referred to \cite{Marcus1999}, \cite{Pes10}, \cite{Ros91}, \cite{Yan} and the references therein.

In this paper, we prove two limit laws for functionals of self-intersection symmetric $\alpha$-stable processes with $\alpha\in(1,2)$. However, when $\alpha=1$, the self-intersection local time for symmetric $\alpha$-stable process does not exist. If we set different normalization factors, we should also get two limit theorems similar to Theorems \ref{Thm-first-order} and \ref{Thm-second-order}, which will be considered in our next paper.

The rest of this paper   is organized as follows. We introduce some preliminary results about the symmetric $\alpha$-stable processes, show the existence of self-intersection local time $L_t$ and give some useful lemmas in Section \ref{sec2}. In Section \ref{sec3}, we give the proofs of the main results. Throughout this paper, if not mentioned otherwise, the letter $c$, with or without a subscript, denotes a generic positive finite constant and may change from line to line.

\section{Preliminaries}\label{sec2}

 In this section, we first recall the definition and properties of  $\alpha$-stable process. For background on all these notions, we refer to \cite{Applebaum}, \cite{Samorodnitsky}, \cite{Pip2017} and the  references therein.  Throughout this paper we fix a complete probability space $(\Omega,\mathcal{F},P, \mathcal{F}_{t})$ such that the processes considered are well-defined on the space.

Let the parameters $\alpha,\sigma,\beta,\mu$ satisfy
$$
\alpha\in(0,2],\quad\sigma\in(0,+\infty),\quad\beta\in[-1,1],\quad\mu\in(-\infty,+\infty),
$$
and denote
\begin{align*}
\phi_{\alpha}(u)=
\begin{cases}
-\sigma^{\alpha}|u|^{\alpha}\Big(1-\iota\beta \text{sgn}(u)\tan\frac{\alpha\pi}{2}\Big)+\iota\mu u,\ \ \ \ \ \  \  \  \alpha\neq1,\\
-\sigma|u|\Big(1+\iota\beta \frac{2}{\pi}\text{sgn}(u)\log|u|\Big)+\iota\mu u,\ \ \ \ \ \ \ \ \   \ \alpha=1,\\
\end{cases}
\end{align*}
with  $\iota^{2}=-1$. A random variable $\eta$ is said to have an $\alpha$-stable distribution, denoted by $\eta\sim S_{\alpha}(\sigma,\beta,\mu)$, if it has the characteristic function
$$\mathbb{E}e^{iu\eta}=e^{-\phi_{\alpha}(u)}.$$
These parameters $\alpha,\sigma,\beta,\mu$ are called the stability index, scale index, skewness index and location index, respectively. When $\mu=0$ and $\beta=0$, we say $\eta$ is symmetrically $\alpha$-stable,
its characteristic function is given by
\begin{equation*}
\mathbb{E}e^{\iota\theta \eta}=e^{-\sigma^{\alpha}|\theta|^{\alpha}},\qquad \theta\in\mathbb{R}.
\end{equation*}
for some scale parameter $\sigma>0$.
For any $\eta\sim S_{\alpha}(\sigma,\beta,\mu)$ with $0<\alpha<2$, we then have $\mathbb{E}|\eta|^{p}<\infty$ for all $0<p
<\alpha$ and $\mathbb{E}|\eta|^{p}=\infty$ for all $p\geq\alpha$. Moreover, if $\xi_{n}\sim S_{\alpha}(\sigma_n,\beta,\mu)$ for every $n\geq1$ and $\sigma_{n}\rightarrow\sigma>0~ (n\rightarrow\infty)$, then $\xi_{n}$ converges to an $\alpha$-stable random variable $\xi\sim S_{\alpha}(\sigma,\beta,\mu)$ in distribution, as $n$ tends to infinity.
An $\{\mathcal{F}_{t}\}$-adapted process $X=\{X_{t},t\geq0\}$ with all sample paths in $D[0,\infty]$ is said to be an $\alpha$-stable process  with $\alpha\in(0,2]$ if for any $t> s\geq0$,
$$\mathbb{E}[e^{\iota u(X_{t}-X_{s})}|\mathcal{F}_{s}]=e^{-(t-s)\phi_{\alpha}(u)},\quad u\in\mathbb{R},$$
where $\phi_{\alpha}(u)$ is called the L\'{e}vy symbol of $X$.  When $\mu=\beta=0$, the $\alpha$-stable process is called
 symmetric  $\alpha$-stable process and for any $t>s\geq0$,

\begin{equation}\label{sec2-eq2.2}
\mathbb{E}[e^{\iota u(X_t-X_s)}|\mathcal{F}_{s}]=e^{-(t-s)|u|^{\alpha}},\qquad u\in\mathbb{R}
\end{equation}
 where $X_0=0$.
More results for symmetric $\alpha$-stable process can be found in Bass and Khoshnevisan \cite{Bass1993}, Sun and Yan \cite{Sun2021},  Roger and Walsh \cite{Rog1991a, Rog1991b} and references therein.

Next, we will show the existence of self-intersection local time $L_t$.
By the definition of self-intersection local time of $X$ in \eqref{sec1-eq1.2}, we can set
$$f_\varepsilon(x)=\frac1{\sqrt{2\pi\varepsilon}}e^{-\frac{|x|^2}{2\varepsilon}}=\frac1{2\pi}\int_{\mathbb{R}}e^{\iota px}e^{-\varepsilon \frac{|p|^2}{2}}dp.$$
Since the Dirac delta function $\delta$ can be approximated by $f_\varepsilon(x)$, we can approximate  $L_t(x)$ by
\begin{equation}\label{sec2-eq2.3}
\widehat{L}_{t,\varepsilon}(x):=\int_{0<r<s<t}f_\varepsilon(X^H_s-X^H_r-x)drds, ~~\text{as}~~\varepsilon\to0.
\end{equation}

If $\widehat{L}_{t,\varepsilon}(x)$ converges to a random variable in $L^2$ as $\varepsilon\to0$, we denote the limit by $L_t(x)$ and call the self-intersection local time of $X$ exists in $L^2$.
\begin{proposition}\label{sec2-pro2.1}
If $1<\alpha<2$, then the self-intersection local time $L_t(x)$ exists in $L^2$.
\end{proposition}
\begin{proof}
Let $\varepsilon>0$. By \eqref{sec2-eq2.3}, we have
\begin{align*}
\mathbb{E}|\widehat{L}_{t,\varepsilon}(x)|^2= \frac1{4\pi^2}\int_{\{0<r<s<t\}^2}\int_{\mathbb{R}^2}
e^{-\frac{\varepsilon}{2}(|p_1|^2+|p_2|^2)}\mathbb{E}e^{-\iota p_1(X_{s_1}-X_{r_1}-x)}e^{-\iota p_2(X_{s_2}-X_{r_2}-x)}dpdrds.
\end{align*}

For the set $\{0<r_1<s_1<t\}\times\{0<r_2<s_2<t\}$, there are six possibilities for the order of $r_j$ and $s_j$, $j=1,2$. Because of symmetry, we only need to consider three sets $E_{t,1}=\{0<r_1<s_1<r_2<s_2<t\}$, $E_{t,2}=\{0<r_1<r_2<s_1<s_2<t\}$ and $E_{t,3}=\{0<r_1<r_2<s_2<s_1<t\}$. Then we have
\begin{align*}
\mathbb{E}|\widehat{L}_{t,\varepsilon}(x)|^2
&\leq c\,\int_{E_{t,1}+E_{t,2}+E_{t,3}}\int_{\mathbb{R}^2}
\mathbb{E}e^{-\iota p_1(X_{s_1}-X_{r_1})}e^{-\iota p_2(X_{s_2}-X_{r_2})}dpdrds\\
&=:c(\Lambda_1+\Lambda_2+\Lambda_3).
\end{align*}
Similarly, for any $\varepsilon,\eta>0$, we have
\begin{align*}
	\mathbb{E}|\widehat{L}_{t,\varepsilon}(x)-\widehat{L}_{t,\eta}(x)|^2\le c(\Lambda_1+\Lambda_2+\Lambda_3).
\end{align*}
So we only need to show $\Lambda_j<\infty$, $j=1,2,3$, under the condition $\alpha>1$.

We first consider $\Lambda_1$. By \eqref{sec2-eq2.2}, on the set $E_{t,1}$,
$$\mathbb{E}e^{-\iota p_1(X_{s_1}-X_{r_1})}e^{-\iota p_2(X_{s_2}-X_{r_2})}=e^{-|p_1|^\alpha(s_1-r_1)-|p_2|^\alpha(s_2-r_2)}.$$
Then
\begin{align*}
\Lambda_1&\leq c\,\int_{E_{t,1}}\int_{\mathbb{R}^2}e^{-|p_1|^\alpha(s_1-r_1)-|p_2|^\alpha(s_2-r_2)}dpdrds\\
&\leq c\,\int_{E_{t,1}}(s_1-r_1)^{-\frac1\alpha}(s_2-r_2)^{-\frac1\alpha}drds,
\end{align*}
which is finite under condition $\alpha>1$.

Similarly, on the set $E_{t,2}$,
$$\mathbb{E}e^{-\iota p_1(X_{s_1}-X_{r_1})}e^{-\iota p_2(X_{s_2}-X_{r_2})}=e^{-|p_1|^\alpha(r_2-r_1)-|p_1+p_2|^\alpha(s_1-r_2)-|p_2|^\alpha(s_2-s_1)};$$
on the set $E_{t,3}$,
$$\mathbb{E}e^{-\iota p_1(X_{s_1}-X_{r_1})}e^{-\iota p_2(X_{s_2}-X_{r_2})}=e^{-|p_1|^\alpha(s_1-s_2+r_2-r_1)-|p_1+p_2|^\alpha(s_2-r_2)}.$$
Then
\begin{align*}
\Lambda_2&\leq c\,\int_{E_{t,2}}\int_{\mathbb{R}^2}e^{-|p_1|^\alpha(r_2-r_1)-|p_1+p_2|^\alpha(s_1-r_2)-|p_2|^\alpha(s_2-s_1)}dpdrds\\
&\leq c\,\int_{E_{t,2}}(r_2-r_1)^{-\frac1\alpha}(s_2-s_1)^{-\frac1\alpha}drds
\end{align*}
and
\begin{align*}
\Lambda_3&\leq c\,\int_{E_{t,3}}\int_{\mathbb{R}^2}e^{-|p_1|^\alpha(s_1-s_2+r_2-r_1)-|p_1+p_2|^\alpha(s_2-r_2)}dpdrds\\
&\leq c\,\int_{E_{t,3}}(s_1-s_2)^{-\frac1\alpha}(s_2-r_2)^{-\frac1\alpha}drds.
\end{align*}
Thus, we obtain the desired results under the condition $\alpha>1$.
\end{proof}

At the end of this section, we give some useful lemmas to prove the theorems in Section \ref{sec3}. Lemmas \ref{sec2-lem2-Ltm} will be used to prove Theorem  \ref{Thm-first-order},
the result of Lemmas \ref{sec2-lem1-m}-\ref{sec2-lem1-m'} plays a key role in the proof of Theorem  \ref{Thm-second-order}.

\begin{lemma}\label{sec2-lem2-Ltm}
Let $\ell_{0}=\ell_{0}(u,v)=0$ and $1<\alpha<2$. For any $m\in \mathbb{N}$, and $t\geq0$, we have
\begin{align*}
\mathbb{E}[(L_t(0))^m]=\frac{1}{(2\pi)^m}\int_{\{0<v<u<t\}^m}\int_{\mathbb{R}^{m}}\prod_{j=1}^{2m}e^{-|\sum\limits_{p=j}^{2m}y_p(x)|^{\alpha}(\ell_j(u,v)-\ell_{j-1}(u,v))}dxdvdu,
\end{align*}
where $\ell_{j}=\ell_{j}(u,v)$, $j=1,2,\cdots,2m$ are functions of $v_1, u_1, \cdots, v_m, u_m$ satisfying that  $\ell_1(u,v)\leq\ell_2(u,v)\leq\cdots\leq\ell_{2m}(u,v)$ is a relabeling of $\{v_1, u_1, \cdots, v_m, u_m\}$ and  $y_{j}=y_{j}(x)$, $j=1,2,\cdots,2m$ are functions of $x_1,\cdots,x_m$ satisfying
$$y_{j}(x)=\left\{\begin{array}{cl}
x_i, & \text{if} \,\,\ell_j=u_i, \\ -x_i, & \text{if} \,\,\ell_j=v_i,
\end{array}\right.$$ for $i=1,2, \cdots, m$ and $j=1,2,\cdots,2m$.
\end{lemma}
\begin{proof}
Let $\varepsilon>0$. By Fourier transform, we can see
\begin{align*}
\mathbb{E}[(\widehat{L}_{t,\varepsilon}(0))^m]&=\frac{1}{(2\pi)^m}\mathbb{E}\left(\int_0^t\int_0^u\int_{\mathbb{R}}e^{-\frac{\varepsilon}{2}|x|^2}e^{\iota x(X_u-X_v)}dxdvdu\right)^m\\
&=\frac{1}{(2\pi)^m}\int_{\{0<v<u<t\}^m}\int_{\mathbb{R}^{m}}\mathbb{E}\prod_{j=1}^me^{-\frac{\varepsilon}{2}|x_j|^2}e^{\iota x_j(X_{u_j}-X_{v_j})}dxdvdu\\
&=\frac{1}{(2\pi)^m}\int_{\{0<v<u<t\}^m}\int_{\mathbb{R}^{m}}\prod_{j=1}^me^{-\frac{\varepsilon}{2}|x_j|^2}\mathbb{E}\prod_{j=1}^{2m}e^{\iota (\sum\limits_{p=j}^{2m}y_p(x))\cdot(X_{\ell_j(u,v)}-X_{\ell_{j-1}(u,v)})}dxdvdu\\
&=\frac{1}{(2\pi)^m}\int_{\{0<v<u<t\}^m}\int_{\mathbb{R}^{m}}\prod_{j=1}^me^{-\frac{\varepsilon}{2}|x_j|^2}\prod_{j=1}^{2m}e^{-|\sum\limits_{p=j}^{2m}y_p(x)|^{\alpha}(\ell_j(u,v)-\ell_{j-1}(u,v))}dxdvdu,
\end{align*}
where in the last equality we use \eqref{sec2-eq2.2}. So that we only need to prove that
\begin{align*}
	\Lambda^{(m)}:=\int_{\{0<v<u<t\}^m}\int_{\mathbb{R}^{m}}\prod_{j=1}^{2m}e^{-|\sum\limits_{p=j}^{2m}y_p(x)|^{\alpha}(\ell_j(u,v)-\ell_{j-1}(u,v))}dxdvdu<\infty.
\end{align*}
Denote $A(u,v)$ as a subset of $\{1,2, \cdots, 2m\}$ such that if $j\in A(u,v)$, we have $y_j(x)=x_i$ for some $i\in\{1,2,\cdots,m\}$. It is easy to get that $A(u,v)$ contains $m$ elements and 
\begin{align}\label{Auv}
	A(u,v)=\{J(1),J(2),\cdots,J(m)\}\subseteq\{1,2, \cdots, 2m\}
\end{align} with $J(1)>J(2)>\cdots>J(m)$. Then we have
\begin{align*}
\Lambda^{(m)}\le\int_{\{0<v<u<t\}^m}\int_{\mathbb{R}^{m}}\prod_{i=1}^{m}e^{-|\sum\limits_{p=J(i)}^{2m}y_p(x)|^{\alpha}(\ell_{J(i)}(u,v)-\ell_{J(i)-1}(u,v))}dxdvdu.
\end{align*}
Let $J'(1)<\cdots<J'(m)$ be the numbers satisfying $J'(i)\notin A(u,v)$ and $J'(i)\in\{1,2,\cdots,2m\}$ for $i=1,2,\cdots m$ and denote
\begin{align}\label{A'uv}
A'(u,v)=\{J'(1),J'(2),\cdots,J'(m)\}.
\end{align}
Because we have $v_j<u_j$, $j=1,2,\cdots,m$, there are $\frac{(2m)!}{2^m}$ possibilities of permutations for $\{v_1, u_1, \cdots, v_m, u_m\}$. We divide the region $\{0<v<u<t\}^m$ by these $\frac{(2m)!}{2^m}$ permutations.  So
 $\ell_{j}(u,v)$ and $\sum\limits_{p=j}^{2m}y_p(x)$, $j=1,\cdots,2m$, will not change on each piece, where using coordinate tranform $z_i=\sum\limits_{p=J(i)}^{2m}y_p(x)$ and $u_i=\ell_{J(i)}(u,v)-\ell_{J(i)-1}(u,v)$ for $i=1,\cdots,m$, and $u_i=\ell_{J'(i)}(u,v)-\ell_{J'(i)-1}(u,v)$ for $i=m+1,\cdots,2m$, we get
\begin{align*}
\Lambda^{(m)}&\le\frac{(2m)!}{2^m}\int_{\{\sum_{i=1}^{2m}u_i<nt, u_i>0, i=1,2,\cdots,2m\}}\int_{\mathbb{R}^{m}}\prod_{i=1}^{m}e^{-|z_i|^{\alpha}u_i}dzdu.
\end{align*}
So denote
\begin{align}\label{O2m}
	O_{2m}=\{(u_1,u_2,\cdots, u_{2m}): ~\sum_{i=1}^{2m}u_i<nt, u_i>0, i=1,2,\cdots,2m\},
\end{align}
 we have
\begin{align*}
	\Lambda^{(m)}\le\frac{(2m)!}{2^m}\left(\int_{\mathbb{R}}e^{-|x|^\alpha}dx\right)^m\int_{O_{2m}}\prod_{j=1}^{m}u_j^{-\frac{1}{\alpha}}du.
\end{align*}
Then we get $\Lambda^{(m)}<\infty$ because $\frac{1}{\alpha}<1$, which completes the proof.
\end{proof}
\begin{lemma}\label{sec2-lem1-m}
	Let $m$ be even and $O_{2m}$ be that in \eqref{O2m}. If $f$ satisfies $\int_{\mathbb{R}}|xf(x)|dx<\infty$ and $\int_{\mathbb{R}}f(x)dx=0$, for $1<\alpha<2$,
	\begin{align*}
	\lim_{n\to\infty}\frac{1}{n^{m\frac{3\alpha-1}{2\alpha}}}\int_{\mathbb{R}^m}\int_{O_{2m}}\prod_{j=1}^{\frac{m}{2}}|\widehat{f}(z_{2j-1})|^2\prod_{j=1}^{m}e^{-|z_j|^\alpha u_j}&du dz \\=\frac{\Gamma^{\frac{m}{2}}(1-\frac1\alpha)\left(\int_{\mathbb{R}}e^{-|x|^\alpha}dx\right)^{\frac{m}{2}}}{\Gamma(\frac{3m}{2}+1-\frac{m}{2\alpha})}&t^{\frac{3m}{2}-\frac{m}{2\alpha}}\times\left(\int_{\mathbb{R}}|\widehat{f}(x)|^2|x|^{-\alpha}dx\right)^{\frac{m}{2}},
	\end{align*}
	where $du=du_1\cdots du_{2m}$, $dz=dz_1\cdots dz_m$ and $\Gamma(\cdot)$ denotes the Gamma function.
\end{lemma}

\begin{proof}
	We first have
	\begin{align*}
\frac{1}{n^{m\frac{3\alpha-1}{2\alpha}}}\int_{\mathbb{R}^m}\int_{O_{2m}}\prod_{j=1}^{\frac{m}{2}}|\widehat{f}(z_{2j-1})|^2\prod_{j=1}^{m}e^{-|z_j|^\alpha u_j}&du dz\\
=\frac{1}{n^{m\frac{3\alpha-1}{2\alpha}}}&\int_{\mathbb{R}^m}\int_{O_{2m}}\prod_{j=1}^{\frac{m}{2}}|\widehat{f}(z_{j})|^2\prod_{j=1}^{m}e^{-|z_j|^\alpha u_j}du dz.
	\end{align*}
	Because for $0\le a_j<1$, $j=1,2,\cdots,m$ and $t>0$, by Lemma 4.1 in \cite{hx}, we have $$\int_{\{\sum_{j=1}^{m}u_j<t,u_j>0,j=1,\cdots,m\}}\prod\limits_{j=1}^{m}u_j^{-a_j}du=\frac{\prod\limits_{j=1}^{m}\Gamma(1-a_j)}{\Gamma\big(1+\sum\limits_{j=1}^{m}(1-a_j)\big)}t^{\sum_{j=1}^{m}(1-a_j)},$$
	So performing integration with respect to $z_{\frac{m}{2}+1},\cdots,z_{m}$ and then with respect to $u_{\frac{m}{2}+1},\cdots,u_{2m}$, we get
	\begin{align*}
&\frac{1}{n^{m\frac{3\alpha-1}{2\alpha}}}\int_{\mathbb{R}^m}\int_{O_{2m}}\prod_{j=1}^{\frac{m}{2}}|\widehat{f}(z_{j})|^2\prod_{j=1}^{m}e^{-|z_j|^\alpha u_j}du dz\\
&\quad=\frac{\left(\int_{\mathbb{R}}e^{-|x|^\alpha}dx\right)^{\frac{m}{2}}}{n^{m\frac{3\alpha-1}{2\alpha}}}\int_{O_{\frac{m}{2}}}\int_{\{\sum\limits_{j=\frac{m}{2}+1}^{2m}u_j<nt-\sum\limits_{j=1}^{\frac{m}{2}}u_{j},u_j>0,j=\frac{m}{2}+1,\cdots,2m\}}\prod_{j=\frac{m}{2}+1}^{m}u_j^{-\frac{1}{\alpha}}\\
&\qquad\qquad\qquad\qquad\qquad\times\int_{\mathbb{R}^{\frac{m}{2}}}\prod_{j=1}^{\frac{m}{2}}|\widehat{f}(z_{j})|^2e^{-|z_{j}|^\alpha u_{j}}dzdu\\
&\quad=\frac{\left(\int_{\mathbb{R}}e^{-|x|^\alpha}dx\right)^{\frac{m}{2}}}{n^{m\frac{3\alpha-1}{2\alpha}}}\int_{O_{\frac{m}{2}}}\frac{\Gamma^{\frac{m}{2}}(1-\frac1\alpha)}{\Gamma(\frac{3m}{2}+1-\frac{m}{2\alpha})}\Big(nt-\sum_{j=1}^{\frac{m}{2}}u_{j}\Big)^{m\frac{3\alpha-1}{2\alpha}}\int_{\mathbb{R}^{\frac{m}{2}}}\prod_{j=1}^{\frac{m}{2}}|\widehat{f}(z_{j})|^2e^{-|z_{j}|^\alpha u_{j}}dzdu\\
&\quad=\frac{\Gamma^{\frac{m}{2}}(1-\frac1\alpha)\left(\int_{\mathbb{R}}e^{-|x|^\alpha}dx\right)^{\frac{m}{2}}}{\Gamma(\frac{3m}{2}+1-\frac{m}{2\alpha})}\int_{O_{\frac{m}{2}}}\Big(t-\frac{1}{n}\sum_{j=1}^{\frac{m}{2}}u_{j}\Big)^{m\frac{3\alpha-1}{2\alpha}}\int_{\mathbb{R}^{\frac{m}{2}}}\prod_{j=1}^{\frac{m}{2}}|\widehat{f}(z_{j})|^2e^{-|z_{j}|^\alpha u_{j}}dzdu.
	\end{align*}
Thus, by
$$\int_{(0,\infty)^{\frac{m}{2}}}\int_{\mathbb{R}^{\frac{m}{2}}}\prod_{j=1}^{\frac{m}{2}}|\widehat{f}(z_{j})|^2e^{-|z_{j}|^\alpha u_{j}}dzdu=\left(\int_0^{\infty}\int_{\mathbb{R}}|\widehat{f}(z)|^2e^{-|z|^{\alpha}u}dzdu\right)^{\frac{m}{2}}<\infty,$$
we get the desired result by letting $n\to\infty$ and using dominated convergence theorem.
\end{proof}

\begin{lemma}\label{sec2-lem1-m'}
	Let $m$ be even and $O_{2m}$ be that in \eqref{O2m}. If $f$ satisfies $\int_{\mathbb{R}}|xf(x)|dx<\infty$ and $\int_{\mathbb{R}}f(x)dx=0$, for $1<\alpha<2$, we have
	\begin{equation}\label{Lambdan}
		\begin{split}
		\lim_{n\to\infty}\Lambda_f(n):=\lim_{n\to\infty}\frac{1}{n^{m\frac{3\alpha-1}{2\alpha}}}&\int_{\mathbb{R}^m}\int_{O_{2m}}\prod_{j=1}^{\frac{m}{2}}|\widehat{f}(z_{2j-1})|^2\prod_{j=1}^{m}e^{-|z_j|^\alpha u_j}\\&\times e^{-\min\{|z_1+z_2|^{\alpha},|z_1+z_2+z_3|^{\alpha},|z_1+z_2+z_4|^{\alpha}\}u_{m+1}} du dz=0.
		\end{split}
	\end{equation}
\end{lemma}

\begin{proof} We first have
\begin{align*}
\Lambda_f(n)&\le\frac{1}{n^{m\frac{3\alpha-1}{2\alpha}}}\int_{\mathbb{R}^m}\int_{[0,nt]^{2m}}\prod_{j=1}^{\frac{m}{2}}|\widehat{f}(z_{2j-1})|^2\prod_{j=1}^{m}e^{-|z_j|^\alpha u_j} e^{-|z_1+z_2|^{\alpha}u_{m+1}} du dz\\
&\qquad+\frac{1}{n^{m\frac{3\alpha-1}{2\alpha}}}\int_{\mathbb{R}^m}\int_{[0,nt]^{2m}}\prod_{j=1}^{\frac{m}{2}}|\widehat{f}(z_{2j-1})|^2\prod_{j=1}^{m}e^{-|z_j|^\alpha u_j}e^{-|z_1+z_2+z_3|^{\alpha}u_{m+1}} du dz\\
&\qquad+\frac{1}{n^{m\frac{3\alpha-1}{2\alpha}}}\int_{\mathbb{R}^m}\int_{[0,nt]^{2m}}\prod_{j=1}^{\frac{m}{2}}|\widehat{f}(z_{2j-1})|^2\prod_{j=1}^{m}e^{-|z_j|^\alpha u_j}e^{-|z_1+z_2+z_4|^{\alpha}u_{m+1}} du dz\\
&=:\Lambda_f^{(1)}(n)+\Lambda_f^{(2)}(n)+\Lambda_f^{(3)}(n).
		\end{align*}
Because of the fact that  $$\int_0^{nt}\int_{\mathbb{R}}|\widehat{f}(z)|^2e^{-|z|^{\alpha}u}dzdu<\int_0^{\infty}\int_{\mathbb{R}}|\widehat{f}(z)|^2e^{-|z|^{\alpha}u}dzdu<\infty$$ and $$\frac{1}{n^{1-\frac{1}{\alpha}}}\int_0^{nt}\int_{\mathbb{R}}e^{-|z|^{\alpha}u}dzdu=\int_{\mathbb{R}}e^{-|z|^\alpha}dz\times\int_{0}^{t}u^{-\frac{1}{\alpha}}du<\infty,$$
we get that for $\Lambda_f^{(1)}(n)$,
\begin{align*}
	\Lambda_f^{(1)}(n)&\le\frac{c}{n^{2-\frac{1}{\alpha}}}\int_{\mathbb{R}^2}\int_{[0,nt]^3}|\widehat{f}(z_{1})|^2\prod_{j=1}^{2}e^{-|z_j|^\alpha u_j} e^{-|z_1+z_2|^{\alpha}u_{m+1}}dudz\\&\le c\int_{\mathbb{R}^2}\int_0^{\infty}\int_{[0,t]^2}|\widehat{f}(z_{1})|^2\prod_{j=1}^{2}e^{-|z_j|^\alpha u_j}e^{-n|z_1+z_2/n^{\frac{1}{\alpha}}|^{\alpha}u_{m+1}}du_2du_{m+1}du_1dz.
\end{align*}
where we change the coordinates $(u_{m+1}, u_2, z_2)$ by $(u_{m+1}n^{-1}, u_2n^{-1}, z_2n^{\frac1{\alpha}})$ in the second inequality.

Similarly, we have
\begin{align*}
	\Lambda_f^{(2)}(n)&\le\frac{c}{n^{2-\frac{1}{\alpha}}}\int_{\mathbb{R}^3}\int_{[0,nt]^4}|\widehat{f}(z_{1})|^2|\widehat{f}(z_{3})|^2\prod_{j=1}^{3}e^{-|z_j|^\alpha u_j}e^{-|z_1+z_2+z_3|^{\alpha}u_{m+1}}dudz\\&\le c\int_{\mathbb{R}^3}\int_{(0,\infty)^2}\int_{[0,t]^2}|\widehat{f}(z_{1})|^2|\widehat{f}(z_{3})|^2\prod_{j=1}^{3}e^{-|z_j|^\alpha u_j}e^{-n|z_1+z_2/n^{\frac{1}{\alpha}}+z_3|^{\alpha}u_{m+1}}du_2du_{m+1}du_1du_3dz,
\end{align*}
and
\begin{align*}
	\Lambda_f^{(3)}(n)&\le\frac{c}{n^{3-\frac{2}{\alpha}}}\int_{\mathbb{R}^3}\int_{[0,nt]^4}|\widehat{f}(z_{1})|^2\prod_{j=1}^{3}e^{-|z_j|^\alpha u_j}e^{-|z_1+z_2+z_3|^{\alpha}u_{m+1}}dudz\\&\le c\int_{\mathbb{R}^3}\int_0^{\infty}\int_{[0,t]^3}|\widehat{f}(z_{1})|^2\prod_{j=1}^{3}e^{-|z_j|^\alpha u_j}e^{-n|z_1+z_2/n^{\frac{1}{\alpha}}+z_3/n^{\frac{1}{\alpha}}|^{\alpha}u_{m+1}}du_2du_3du_{m+1}du_1dz.
\end{align*}
Then by dominated convergence theorem, we have $\lim\limits_{n\to\infty}\Lambda_f(n)=0$, since
$$\int_0^{\infty}\int_{\mathbb{R}}|\widehat{f}(z)|^2e^{-|z|^{\alpha}u}dzdu<\infty\text{ and }\int_0^{t}\int_{\mathbb{R}}e^{-|z|^{\alpha}u}dzdu<\infty.$$
\end{proof}

\section{Proof of the main results}\label{sec3}
In this section, we will give the proof of main results using preliminary results in Section \ref{sec2}.

\subsection{Proof of Theorem \ref{Thm-first-order}}\label{sec3.1}
For convenience, let
$$F_n(t):=\frac{1}{n^{2-\frac1\alpha}}\int_0^{nt}\int_0^uf(X_u-X_v)dvdu.$$
We need to show the limit law for $F_n(t)$, as $n\to\infty$, in this subsection.
By Fourier transform,
\begin{align*}
F_n(t)=\frac{1}{2\pi n^{2-\frac1\alpha}}\int_0^{nt}\int_0^u\int_{\mathbb{R}}\widehat{f}(x)e^{\iota x(X_u-X_v)}dxdvdu,
\end{align*}
where $\iota^2=-1$ and $\widehat{f}(x)=\int_{\mathbb{R}}f(\xi)e^{-\xi x}d\xi$ denotes the Fourier transform of $f(x)$.

Let $$G_n(t)=\frac{1}{2\pi n^{2-\frac1\alpha}}\int_0^{nt}\int_0^u\int_{\mathbb{R}}\widehat{f}(0)e^{\iota x(X_u-X_v)}dxdvdu.$$

We first show the difference of $F_n(t)$ and $G_n(t)$ converges to zero in $L^2$, then we only need to consider the limit law of $G_n(t)$, as $n\to\infty$.

\begin{lemma}\label{sec3.1-lem-FG}
As $n\to\infty$, we have
$$|F_n(t)-G_n(t)|\overset{L^2}{\to}0.$$
\end{lemma}
\begin{proof}
By the definitions of $F_n(t)$ and $G_n(t)$,
\begin{align*}
&\mathbb{E}|F_n(t)-G_n(t)|^2\\
&=\frac{1}{(2\pi n^{2-\frac1\alpha})^2}\int_{D^2_{nt}\times D^2_{nt}}\int_{\mathbb{R}^2}|\widehat{f}(x_1)-\widehat{f}(0)||\widehat{f}(x_2)-\widehat{f}(0)|
\mathbb{E}\left[e^{\iota x_1(X_{u_1}-X_{v_1})}e^{\iota x_2(X_{u_2}-X_{v_2})}\right]dxdvdu\\
&\leq \frac{c}{n^{4-\frac2\alpha}}\int_{E_{nt,1}+E_{nt,2}+E_{nt,3}}\int_{\mathbb{R}^2}|x_1||x_2|
\mathbb{E}\left[e^{\iota x_1(X_{u_1}-X_{v_1})}e^{\iota x_2(X_{u_2}-X_{v_2})}\right]dxdvdu\\
&=:F_{n,1}(t)+F_{n,2}(t)+F_{n,3}(t),
\end{align*}
where $D^2_{nt}=\{0<v<u<nt\}$, $E_{nt,1}=\{0<v_1<u_1<v_2<u_2<nt\}$, $E_{nt,2}=\{0<v_1<v_2<u_1<u_2<nt\}$ and $E_{nt,3}=\{0<v_1<v_2<u_2<u_1<nt\}$.

For $F_{n,1}(t)$, on the set $E_{nt,1}$,
$$\mathbb{E}\left[e^{\iota x_1(X_{u_1}-X_{v_1})}e^{\iota x_2(X_{u_2}-X_{v_2})}\right]=e^{-|x_1|^\alpha(u_1-v_1)-|x_2|^\alpha(u_2-v_2)}.$$
Then
\begin{align*}
F_{n,1}(t)&\leq c\frac{1}{n^{4-\frac2\alpha}}\int_{E_{nt,1}}\int_{\mathbb{R}^2}|x_1||x_2|e^{-|x_1|^\alpha(u_1-v_1)-|x_2|^\alpha(u_2-v_2)}dxdvdu\\
&\leq c\frac{1}{n^{4-\frac2\alpha}}\int_{E_{nt,1}}(u_1-v_1)^{-\frac2\alpha}(u_2-v_2)^{-\frac2\alpha}dvdu\\
&\leq c\frac{1}{n^{4-\frac2\alpha}}\int_0^{nt}\int_0^{u_2}(u_2-v_2)^{-\frac2\alpha}v_2^{2-\frac2\alpha}dv_2du_2\\
&\leq c\,n^{-\frac2\alpha}.
\end{align*}

Similarly, we have
\begin{align*}
F_{n,2}(t)&\leq c\frac{1}{n^{4-\frac2\alpha}}\int_{E_{nt,2}}\int_{\mathbb{R}^2}|x_1||x_2|e^{-|x_1|^\alpha(v_2-v_1)-|x_1+x_2|^{\alpha}(u_1-v_2)-|x_2|^\alpha(u_2-u_1)}dxdvdu\\
&\leq c\frac{1}{n^{4-\frac2\alpha}}\int_{E_{nt,2}}(v_2-v_1)^{-\frac2\alpha}(u_2-u_1)^{-\frac2\alpha}dvdu\\
&\leq c\,n^{-\frac2\alpha}.
\end{align*}
and
\begin{align*}
F_{n,3}(t)&\leq c\frac{1}{n^{4-\frac2\alpha}}\int_{E_{nt,3}}\int_{\mathbb{R}^2}|x_1||x_2|e^{-|x_1|^\alpha(u_1-u_2+v_2-v_1)-|x_1+x_2|^\alpha(u_2-v_2)}dxdvdu\\
&\leq c\frac{1}{n^{4-\frac2\alpha}}\int_{E_{nt,3}}(u_1-u_2)^{-\frac2\alpha}\left(\min\{u_2-v_2,v_2-v_1\}\right)^{-\frac2\alpha}dvdu\\
&\leq c\frac{1}{n^{4-\frac2\alpha}}\int_0^{nt}\int_0^{u_1}(u_1-u_2)^{-\frac2\alpha}u_2^{2-\frac2\alpha}du_2du_1\\
&\leq c\,n^{-\frac2\alpha}.
\end{align*}
Together this three estimates gives the desired result.
\end{proof}

In the following result, we show the limit law of $G_n(t)$, as $n\to\infty$.

\begin{proposition}\label{sec3.1-prop}
Suppose that $f$ is bounded and $\int_{\mathbb{R}}|xf(x)|dx<\infty$. Then for any $t>0$,
$$G_n(t)\overset{law}{=}\left(\int_{\mathbb{R}}f(x)dx\right)L_t(0),$$ where $L_t(0)$ is the self-intersection local time of $X$ at $0$.
\end{proposition}
\begin{proof}
 For any $m\in\mathbb{N}$,
\begin{align*}
\mathbb{E}\left[(G_n(t))^m\right]&=\frac{(\widehat{f}(0))^m}{(2\pi)^mn^{2m-\frac{m}{\alpha}}}\mathbb{E}\left(\int_{\{0<v<u<nt\}}\int_{\mathbb{R}}e^{\iota x(X_u-X_v)}dxdvdu\right)^m\\&=\frac{(\widehat{f}(0))^m}{(2\pi)^mn^{2m-\frac{m}{\alpha}}}\int_{\{0<v<u<nt\}^m}\int_{\mathbb{R}^{m}}\mathbb{E}\prod_{j=1}^me^{\iota x_j(X_{u_j}-X_{v_j})}dxdvdu.
\end{align*}
Using the notations $\ell_{j}=\ell_{j}(u,v)$, $j=0,1,2,\cdots,2m$ and $y_{j}=y_{j}(x)$, $j=1,2,\cdots,2m$ in Lemma \ref{sec2-lem2-Ltm}, we have
\begin{align*}
	\mathbb{E}\left[(G_n(t))^m\right]&=\frac{(\widehat{f}(0))^m}{(2\pi)^mn^{2m-\frac{m}{\alpha}}}\int_{\{0<v<u<nt\}^m}\int_{\mathbb{R}^{m}}\mathbb{E}\prod_{j=1}^{2m}e^{\iota (\sum\limits_{p=j}^{2m}y_p(x))\cdot(X_{\ell_j(u,v)}-X_{\ell_{j-1}(u,v)})}dxdvdu\\&=\frac{(\widehat{f}(0))^m}{(2\pi)^mn^{2m-\frac{m}{\alpha}}}\int_{\{0<v<u<nt\}^m}\int_{\mathbb{R}^{m}}\prod_{j=1}^{2m}e^{-|\sum\limits_{p=j}^{2m}y_p(x)|^{\alpha}(\ell_j(u,v)-\ell_{j-1}(u,v))}dxdvdu\\&=\frac{(\widehat{f}(0))^m}{(2\pi)^m}\int_{\{0<v<u<t\}^m}\int_{\mathbb{R}^{m}}\prod_{j=1}^{2m}e^{-|\sum\limits_{p=j}^{2m}y_p(x)|^{\alpha}(\ell_j(u,v)-\ell_{j-1}(u,v))}dxdvdu,
\end{align*}
where in the last equality we use the fact that $\ell_{j}(nu,nv)=n\ell_{j}(u,v)$ and $y_{j}(\frac{1}{n}x)=\frac{1}{n}y_{j}(x)$, $j=1,2,\cdots,2m$. So that by Lemma \ref{sec2-lem2-Ltm},
\begin{align*}
\mathbb{E}\left[(G_n(t))^m\right]=(\widehat{f}(0))^m\mathbb{E}[(L_t(0))^m]
=\left(\int_{\mathbb{R}}f(x)dx\right)^m\mathbb{E}[(L_t(0))^m].
\end{align*}
Thus, the proposition is proved by the method of moments.
\end{proof}

\subsection{Proof of Theorem \ref{Thm-second-order}}\label{sec3.2}
In this section, we will prove Theorem \ref{Thm-second-order}. To make notations simpler, we will abuse some notations from Section \ref{sec3.2}.

Let $$F_n(t)=\frac{1}{n^{\frac32-\frac{1}{2\alpha}}}\int_0^{nt}\int_0^uf(X_u-X_v)dvdu.$$

By Fourier transform,
\begin{align*}
F_n(t)=\frac{1}{2\pi n^{\frac32-\frac1{2\alpha}}}\int_0^{nt}\int_0^u\int_{\mathbb{R}}\widehat{f}(x)e^{\iota x(X_u-X_v)}dxdvdu.
\end{align*}
Then for $m\in\mathbb{N}$, the moment of $F_n(t)$ is
\begin{align*}
\mathbb{E}(F_n(t))^m=\frac{1}{(2\pi)^m n^{\frac{3m}{2}-\frac{m}{2\alpha}}}\int_{\{0<v<u<nt\}^m}\int_{\mathbb{R}^m}\prod_{j=1}^m\widehat{f}(x_j)\mathbb{E}e^{\iota \sum_{j=1}^mx_j(X_{u_j}-X_{v_j})}dxdvdu.
\end{align*}
Define $\tilde{\mathcal{P}}_{2m}$ to be the set of all permutations of $\{v_1, u_1, \cdots, v_m, u_m\}$ with $v_i$ ahead of $u_i$ for $i=1,2,\cdots,m$. For each $\tilde{\sigma}\in\tilde{\mathcal{P}}_{2m}$, let $\ell_1^{\tilde{\sigma}}\leq\ell_2^{\tilde{\sigma}}\leq\cdots\leq\ell_{2m}^{\tilde{\sigma}}$ be relabeling of $\{v_1, u_1, \cdots, v_m, u_m\}$. On the set $D^{\tilde{\sigma}}_{nt}=\{(\ell_1^{\tilde{\sigma}},\ell_2^{\tilde{\sigma}},\cdots,\ell_{2m}^{\tilde{\sigma}}):0<\ell_1^{\tilde{\sigma}}\leq\ell_2^{\tilde{\sigma}}\leq\cdots\leq\ell_{2m}^{\tilde{\sigma}}<nt\}$, define
$$
y_{j}^{\tilde{\sigma}}=\left\{\begin{array}{cl}
x_i, & \text{if}~\ell_j^{\tilde{\sigma}}=u_i, \\
 -x_i, & \text{if}~\ell_j^{\tilde{\sigma}}=v_i,
\end{array}\right.
$$
for $i=1,2, \cdots, m$ and $j=1,2,\cdots,2m$. Moreover, on $D^{\tilde{\sigma}}_{nt}$ let
\begin{align*}
\{J^{\tilde{\sigma}}(1),J^{\tilde{\sigma}}(2),\cdots,J^{\tilde{\sigma}}(m):J^{\tilde{\sigma}}(1)>J^{\tilde{\sigma}}(2)>\cdots>J^{\tilde{\sigma}}(m)\}\subseteq\{1,2, \cdots, 2m\}
\end{align*}
 satisfying that for each $i\in\{1,2,\cdots,m\}$, we have $y_{J^{\tilde{\sigma}}(i)}^{\tilde{\sigma}}=x_j$ for some $i\in\{1,2,\cdots,m\}$ and let
\begin{align*}
\{\bar{J}^{\tilde{\sigma}}(1),\bar{J}^{\tilde{\sigma}}(2),\cdots,\bar{J}^{\tilde{\sigma}}(m):\bar{J}^{\tilde{\sigma}}(1)>\bar{J}^{\tilde{\sigma}}(2)>\cdots>\bar{J}^{\tilde{\sigma}}(m)\}\subseteq\{1,2, \cdots, 2m\}
\end{align*}
satisfying $$\{\bar{J}^{\tilde{\sigma}}(1),\bar{J}^{\tilde{\sigma}}(2),\cdots,\bar{J}^{\tilde{\sigma}}(m)\}\cap\{J^{\tilde{\sigma}}(1),J^{\tilde{\sigma}}(2),\cdots,J^{\tilde{\sigma}}(m)\}=\emptyset.$$

Here by the defintion of $y^{\tilde{\sigma}}_p$, $p=1,\cdots,m$,  we can get $x_1,\cdots,x_m$ are linear combinations of $\sum\limits_{p=J^{\tilde{\sigma}}(1)}^{2m}y^{\tilde{\sigma}}_p(x),\cdots,\sum\limits_{p=J^{\tilde{\sigma}}(m)}^{2m}y^{\tilde{\sigma}}_p(x)$ and so that $\sum\limits_{p=\bar{J}^{\tilde{\sigma}}(1)}^{2m}y^{\tilde{\sigma}}_p(x),\cdots,\sum\limits_{p=\bar{J}^{\tilde{\sigma}}(m)}^{2m}y^{\tilde{\sigma}}_p(x)$ are also linear combinations of $\sum\limits_{p=J^{\tilde{\sigma}}(1)}^{2m}y^{\tilde{\sigma}}_p(x),\cdots,\sum\limits_{p=J^{\tilde{\sigma}}(m)}^{2m}y^{\tilde{\sigma}}_p(x)$.

Then letting $z_0=0$ and using coordinate transform $z_i=\sum\limits_{p=J^{\tilde{\sigma}}(i)}^{2m}y^{\tilde{\sigma}}_p(x)$, $u_i=\ell_{J^{\tilde{\sigma}}(i)}^{\tilde{\sigma}}-\ell_{J^{\tilde{\sigma}}(i)-1}^{\tilde{\sigma}}$ for $i=1,\cdots,m$, and $u_i=\ell_{\bar{J}^{\tilde{\sigma}}(i-m)}^{\tilde{\sigma}}-\ell_{\bar{J}^{\tilde{\sigma}}(i-m)-1}^{\tilde{\sigma}}$ for $i=m+1,\cdots,2m$, we have
\begin{equation}\label{EFn}
	\begin{split}
	\mathbb{E}(F_n(t))^m=\frac{1}{(2\pi)^m n^{\frac{3m}{2}-\frac{m}{2\alpha}}}\sum_{\tilde{\sigma}\in\tilde{\mathcal{P}}_{2m}}\int_{O_{2m}}\int_{\mathbb{R}^m}\prod\limits_{j=1}^m\widehat{f}(z_{j}-z_{j-1})e^{-|z_j|^{\alpha}u_j-|Z_j^{\tilde{\sigma}}(z)|^{\alpha}u_{j+m}}dzdu,
	\end{split}
\end{equation}
where $O_{2m}$ is that in \eqref{O2m} and $Z_j^{\tilde{\sigma}}(z)=\sum\limits_{p=\bar{J}^{\tilde{\sigma}}(j)}^{2m}y^{\tilde{\sigma}}_p(x)$, $j=1,\cdots,m$ are linear combinations of $z_1=\sum\limits_{p=J^{\tilde{\sigma}}(1)}^{2m}y^{\tilde{\sigma}}_p(x),\cdots,z_m=\sum\limits_{p=J^{\tilde{\sigma}}(m)}^{2m}y^{\tilde{\sigma}}_p(x)$.

To get a bound for $\prod\limits_{j=1}^m\widehat{f}(z_{j}-z_{j-1})$, let $I_{m,0}(z)=\prod\limits_{j=1}^m\widehat{f}(z_{j}-z_{j-1})$ and for $k=1,2,\cdots,m-1$, define
\begin{align*}
I_{m,k}(z)=
\begin{cases}
\prod\limits_{j=1}^{\frac{k+1}{2}}|\widehat{f}(z_{2j-1})|^2\prod\limits_{j=k+2}^m\widehat{f}(z_{j}-z_{j-1}),\,\,\,\,\,\, \qquad\qquad\qquad \text{if ~$k$ ~is ~odd},\\
\prod\limits_{j=1}^{\frac{k}{2}}|\widehat{f}(z_{2j-1})|^2\widehat{f}(z_{k+1})\prod\limits_{j=k+2}^m\widehat{f}(z_{j}-z_{j-1}),\qquad\qquad \text{if ~$k$ ~is ~even}.
\end{cases}
\end{align*}
We have
\begin{align}\label{chain}
|I_{m,0}(z)-I_{m,m-1}(z)|\leq\sum\limits_{j=1}^{m-1}|I_{m,j-1}(z)-I_{m,j}(z)|
\end{align} and for $k=1,2,\cdots,m-1$, let
\begin{align*}
|J_{m,k-1}^{n}-J_{m,k}^n|:=\frac{1}{(2\pi)^m n^{\frac{3m}{2}-\frac{m}{2\alpha}}}\int_{O_{2m}}\int_{\mathbb{R}^m}|I_{m,k-1}(z)-I_{m,k}(z)|\prod_{j=1}^{m}e^{-|z_j|^\alpha u_j}dzdu,
\end{align*}
and
\begin{align*}
J_{m,m-1}^{n}=\frac{1}{(2\pi)^m n^{\frac{3m}{2}-\frac{m}{2\alpha}}}\int_{O_{2m}}\int_{\mathbb{R}^m}|I_{m,m-1}(z)|\prod_{j=1}^{m}e^{-|z_j|^\alpha u_j}dzdu,
\end{align*}
Next, we use the chaining argument introduced in \cite{NuXu} and show $\sum\limits_{j=1}^{m-1}|J^n_{m,j-1}-J^n_{m,j}|\to0$ and $J^n_{m,m-1}\to0$ (when $m$ is odd), as $n\to\infty$. Then, only $J^n_{m,m-1}$ ($m$ is even) plays a role in the final limit law. In the following part, let $x\wedge y=\min\{x,y\}$ and $x\vee {y}=\max\{x,y\}$ and $\lfloor{m}\rfloor$ be the integer part for $m\ge0$.
\begin{proposition}\label{sec3.2-prop}
	There exist positive constants $c$ and $\gamma$ independent of $n$, such that
	$$\sum_{k=1}^{m-1}|J^n_{m,k-1}-J^n_{m,k}|\leq c\,n^{-\gamma}$$
	and
	\begin{align*}
	|J^n_{m,m-1}|\leq
	\begin{cases}
	c\,n^{-\gamma}, \qquad\qquad \text{if ~$m$ ~is ~odd},\\
	c, \,\,\,\,\,\,\,\,\,\,\qquad\qquad \text{if ~$m$ ~is ~even}.\\
	\end{cases}
	\end{align*}
\end{proposition}
\begin{proof}
We have
\begin{align*}
|I_{m,m-1}(z)|\le
\begin{cases}
\prod\limits_{j=1}^{\frac{m}{2}}|\widehat{f}(z_{2j-1})|^2\le c\prod\limits_{j=1}^{\frac{m}{2}}\left(|z_{2j-1}|\wedge 1\right),\qquad\qquad\qquad\qquad\,\,\,\text{if ~$m$ ~is ~even},\\
\prod\limits_{j=1}^{\frac{m-1}{2}}|\widehat{f}(z_{2j-1})|^2|\widehat{f}(z_{m})|\le c\prod\limits_{j=1}^{\frac{m-1}{2}}\left(|z_{2j-1}|\wedge 1\right)\left(|z_{m}|\wedge 1\right),\,\,\,\,\text{if ~$m$ ~is ~odd}.
\end{cases}
\end{align*}
Here we use $|\widehat{f}(x)|\leq c(|x|\wedge 1)^{\beta}$ for $\beta\in[0,1]$ which can de deduced from $\int_{\mathbb{R}}|xf(x)|dx<\infty$ and $\int_{\mathbb{R}}f(x)dx=0$. Because $\frac{1}{n^{1-\frac{1}{\alpha}}}\int_0^{nt}\int_{\mathbb{R}}e^{-|z|^{\alpha}u}dzdu=\int_{\mathbb{R}}e^{-|z|^\alpha}dz\times\int_{0}^{t}u^{-\frac{1}{\alpha}}du<\infty$
and
$$\int_0^{nt}\int_{\mathbb{R}}\left(|z|\wedge 1\right)e^{-|z|^{\alpha}u}dzdu<\int_0^{\infty}\int_{\mathbb{R}}\left(|z|\wedge 1\right)e^{-|z|^{\alpha}u}dzdu=\int_{\mathbb{R}}\left(|z|\wedge 1\right)|z|^{-\alpha}dz<\infty,$$
we have
\begin{align*}
|J^n_{m,m-1}|&\leq\frac{1}{(2\pi)^m n^{\frac{3m}{2}-\frac{m}{2\alpha}}}\int_{[0,nt]^m}\int_{\mathbb{R}^m}|I_{m,m-1}(z)|\prod_{j=1}^{m}e^{-|z_j|^\alpha u_j}dzdu\\
&\leq
\begin{cases}
c\,n^{-\frac{1}{2}(1-\frac{1}{\alpha})}, \qquad\qquad\text{if ~$m$ ~is ~odd},\\
c, \qquad\qquad\qquad\qquad \text{if ~$m$ ~is ~even}.\\
\end{cases}
\end{align*}
Then we turn to $|I_{m,k-1}(z)-I_{m,k}(z)|$'s. If $k$ is even, we have
\begin{align*}
|I_{m,k-1}(z)-I_{m,k}(z)|&\leq \prod\limits_{j=1}^{\frac{k}{2}}|\widehat{f}(z_{2j-1})|^2|\widehat{f}(z_{k+1})-\widehat{f}(z_{k+1}-z_{k})|\prod\limits_{j=k+2}^m|\widehat{f}(z_{j}-z_{j-1})|\\
&\leq c\,\prod\limits_{j=1}^{\frac{k}{2}}(|z_{2j-1}|\wedge1)(|z_{k}|\wedge1)\prod_{j=k+2}^{m}\left((|z_{j}|\wedge1)+(|z_{j-1}|\wedge1)\right)\\&\leq c\,\sum_{p_{k+2},\cdots,p_m}\prod\limits_{j=1}^{\frac{k}{2}}(|z_{2j-1}|\wedge1)(|z_{k}|\wedge1)\prod_{j=k+1}^{m}(|z_{j}|\wedge1)^{ p_j\vee\bar{p}_{j+1}},
\end{align*}
where $p_{k+2},\cdots,p_m\in\{0,1\}$, ${p}_{k+1}=\bar{p}_{m+1}=0$ and $p_{j}+\bar{p}_{j}=1$ for $j=k+2,\cdots,m$. If $p_j\vee \bar{p}_{j+1}=0$ for $j=k+1,\cdots,m-1$, we have $p_{j+1}\vee \bar{p}_{j+2}=1$ and if $p_j\vee \bar{p}_{j+1}=0$ for $j=k+2,\cdots,m$, we have $p_{j-1}\vee \bar{p}_{j}=1$, which leads to the fact that $\sum\limits_{j=k+1}^{m}p_j\vee \bar{p}_{j+1}\le\lfloor\frac{m-k+1}{2}\rfloor$. So that
$$\frac{k}{2}+1+\sum\limits_{j=k+1}^{m}p_j\vee \bar{p}_{j+1}\le\lfloor\frac{m+1}{2}+1\rfloor$$
and we have $|J^n_{m,k-1}-J^n_{m,k}|\le c\,n^{-(1-\frac{1}{\alpha})}$. And if $k$ is odd, we have
\begin{align*}
|I_{m,k-1}(z)-&I_{m,k}(z)|\leq \prod\limits_{j=1}^{\frac{k-1}{2}}|\widehat{f}(z_{2j-1})|^2||\widehat{f}(z_{k})||\widehat{f}(-z_{k})-\widehat{f}(z_{k+1}-z_{k})|\prod\limits_{j=k+2}^m|\widehat{f}(z_{j}-z_{j-1})|\\
&\leq c\,\prod\limits_{j=1}^{\frac{k-1}{2}}(|z_{2j-1}|\wedge1)(|z_{k}|\wedge1)(|z_{k+1}|\wedge1)\prod_{j=k+2}^{m}\left((|z_{j}|\wedge1)+(|z_{j-1}|\wedge1)\right)\\&\le c\,\sum_{p_{k+2},\cdots,p_m}\prod\limits_{j=1}^{\frac{k-1}{2}}(|z_{2j-1}|\wedge1)(|z_{k}|\wedge1)\prod_{j=k+1}^{m}(|z_{j}|\wedge1)^{p_j\vee\bar{p}_{j+1}},
\end{align*}
where $p_{k+2},\cdots,p_m\in\{0,1\}$, ${p}_{k+1}=1,\bar{p}_{m+1}=0$ and $p_{j}+\bar{p}_{j}=1$ for $j=k+2,\cdots,m$. Similarly, if $p_j\vee \bar{p}_{j+1}=0$ for $j=k+1,\cdots,m-1$, we have $p_{j+1}\vee \bar{p}_{j+2}=1$ and if $p_j\vee \bar{p}_{j+1}=0$ for $j=k+2,\cdots,m$, we have $p_{j-1}\vee \bar{p}_{j}=1$. Then $\sum\limits_{j=k+1}^{m}p_j\vee \bar{p}_{j+1}\le\lfloor\frac{m-k}{2}\rfloor$ comes from $p_{k+1}\vee \bar{p}_{k+2}=1$. So that
$$\frac{k-1}{2}+1+\sum\limits_{j=k+1}^{m}p_j\le\lfloor\frac{m-1}{2}+1\rfloor$$
and we have $|J^n_{m,k-1}-J^n_{m,k}|\le c\,n^{-\frac{1}{2}(1-\frac{1}{\alpha})}$.
This completes the proof.
\end{proof}

Now we prove the limit properties of $\mathbb{E}(F_n(t))^m$ as $n\to\infty$.

\begin{proposition}\label{sec3.1-prop2}
Suppose that $f$ is bounded, $\int_{\mathbb{R}}|xf(x)|dx<\infty$ and $\int_{\mathbb{R}}f(x)dx=0$. Then for any $t>0$,
$$F_n(t)\overset{law}{\to}\left(\frac1{4\pi^2}\int_{\mathbb{R}}|\widehat{f}(x)|^2|x|^{-\alpha}dx\right)^{\frac12}Z(t),$$
as $n\to\infty$, where $\widehat{f}$ is the Fourier transform of $f$ and $Z(t)$ is a random variable with parameter $t>0$ and $\mathbb{E}(Z(t))^m=\frac{m!\Gamma^{\frac{m}{2}}(1-\frac1\alpha)\left(\int_{\mathbb{R}}e^{-|x|^\alpha}dx\right)^{\frac{m}{2}}}{\Gamma(\frac{3m}{2}+1-\frac{m}{2\alpha})}t^{\frac{3m}{2}-\frac{m}{2\alpha}}$ if $m$ is even, $\mathbb{E}(Z(t))^m=0$ if $m$ is odd.
\end{proposition}
\begin{proof}
	For odd $m$, we have
	\begin{align*}
		\mathbb{E}(F_n(t))^m\le\frac{(2m)!}{2^m}\left(\sum\limits_{j=1}^{m-1}|J^n_{m,j-1}-J^n_{m,j}|+J^{n}_{m,m-1}\right)
	\end{align*}
	because there are $\frac{(2m)!}{2^m}$ elements in $\tilde{\mathcal{P}}_{2m}$. By proposition \ref{sec3.2-prop}, we have $\mathbb{E}(F_n(t))^m\le cn^{-\gamma}$ so that $\lim\limits_{n\to\infty}\mathbb{E}(F_n(t))^m=0$.
	
	Then for even $m$, denote
	\begin{equation}\label{Gn}
	\begin{split}
	G_n=\frac{1}{(2\pi)^m n^{\frac{3m}{2}-\frac{m}{2\alpha}}}\sum_{\tilde{\sigma}\in\tilde{\mathcal{P}}_{2m}}\int_{O_{2m}}\int_{\mathbb{R}^m}I_{m,m-1}(z)\prod_{j=1}^{m}e^{-|z_j|^{\alpha}u_j-|Z_j^{\tilde{\sigma}}(z)|^{\alpha}u_{j+m}}dzdu,
	\end{split}
	\end{equation}
	by proposition \ref{sec3.2-prop}, we have that
	\begin{align*}
		|\mathbb{E}(F_n(t))^m-G_n|\le c\sum\limits_{j=1}^{m-1}|J^n_{m,j-1}-J^n_{m,j}|\le cn^{-\gamma},
	\end{align*}
	so that $\lim\limits_{n\to\infty}|\mathbb{E}(F_n(t))^m-G_n|=0$.	To calculate the limit of $G_n$, we divide $\tilde{\mathcal{P}}_{2m}$ into two parts by $\tilde{\mathcal{P}}_{2m}^0\subseteq\tilde{\mathcal{P}}_{2m}$, s.t. for any $\tilde{\sigma}\in\tilde{\mathcal{P}}_{2m}^0$, we have $$\{J^{\tilde{\sigma}}(1),J^{\tilde{\sigma}}(2),\cdots,J^{\tilde{\sigma}}(m)\}=\{2,4,\cdots,2m\}.$$
	where there are $m!$ elements in $\tilde{\mathcal{P}}_{2m}^0\subseteq\tilde{\mathcal{P}}_{2m}$. For $\tilde{\sigma}\in\tilde{\mathcal{P}}_{2m}^0$, we have  $Z_j^{\tilde{\sigma}}(z)=0$ for all $j=1,\cdots,m$ and for $\tilde{\sigma}\notin\tilde{\mathcal{P}}_{2m}^0$, we have two cases for $Z_1^{\tilde{\sigma}}(z)$, one is that $Z_1^{\tilde{\sigma}}(z)=z_{h}-z_{h-1}$ for some $h\in\{1,2,\cdots,m\}$ and the other is that $Z_1^{\tilde{\sigma}}(z)=z_j-(z_{h}-z_{h-1})$ for some $h,j\in\{1,2,\cdots,m\}$ with $h\neq j$. Then we divide \eqref{Gn} into two parts:
	\begin{align*}
		G_n&=\frac{1}{(2\pi)^m n^{\frac{3m}{2}-\frac{m}{2\alpha}}}\sum_{\tilde{\sigma}\in\tilde{\mathcal{P}}_{2m}^0}\int_{O_{2m}}\int_{\mathbb{R}^m}\prod_{j=1}^{\frac{m}{2}}|\widehat{f}(z_{2j-1})|^2\prod_{j=1}^{m}e^{-|z_j|^{\alpha}u_j}dzdu\\
&\qquad+\frac{1}{(2\pi)^m n^{\frac{3m}{2}-\frac{m}{2\alpha}}}\sum_{\tilde{\sigma}\notin\tilde{\mathcal{P}}_{2m}^0}\int_{O_{2m}}\int_{\mathbb{R}^m}\prod_{j=1}^{\frac{m}{2}}|\widehat{f}(z_{2j-1})|^2\prod_{j=1}^{m}e^{-|z_j|^{\alpha}u_j-|Z_j^{\tilde{\sigma}}(z)|^{\alpha}u_{j+m}}dzdu\\
&=:G_{n,1}+G_{n,2}.
	\end{align*}
	By Lemma \ref{sec2-lem1-m}, we can see $$\lim\limits_{n\to\infty}G_{n,1}=\frac{m!\Gamma^{\frac{m}{2}}(1-\frac1\alpha)\left(\int_{\mathbb{R}}e^{-|x|^\alpha}dx\right)^{\frac{m}{2}}}{\Gamma(\frac{3m}{2}+1-\frac{m}{2\alpha})}t^{\frac{3m}{2}-\frac{m}{2\alpha}}\times\left(\frac1{4\pi^2}\int_{\mathbb{R}}|\widehat{f}(x)|^2|x|^{-\alpha}dx\right)^{\frac{m}{2}}.$$ Becasuse
	\begin{align*}
		G_{n,2}&\le\frac{1}{(2\pi)^m n^{\frac{3m}{2}-\frac{m}{2\alpha}}}\sum_{\tilde{\sigma}\notin\tilde{\mathcal{P}}_{2m}^0}\int_{O_{2m}}\int_{\mathbb{R}^m}\prod_{j=1}^{\frac{m}{2}}|\widehat{f}(z_{2j-1})|^2\prod_{j=1}^{m}e^{-|z_j|^{\alpha}u_j}e^{-|Z_1^{\tilde{\sigma}}(z)|^{\alpha}u_{m+1}}dzdu\\&\le \frac{c}{n^{m\frac{3\alpha-1}{2\alpha}}}\int_{\mathbb{R}^m}\int_{O_{2m}}\prod_{j=1}^{\frac{m}{2}}|\widehat{f}(z_{2j-1})|^2\prod_{j=1}^{m}e^{-|z_j|^\alpha u_j}e^{-\min\{|z_1+z_2|^{\alpha},|z_1+z_2+z_3|^{\alpha},|z_1+z_2+z_4|^{\alpha}\}u_{m+1}}du dz,
	\end{align*}
	by Lemma \ref{sec2-lem1-m'}, we have $\lim\limits_{n\to\infty}G_{n,2}=0$, which completes the proof.

\end{proof}

\bigskip
\textbf{Acknowledgements}
Qian Yu would like to thank Prof. Greg Markowsky for valuable conversations.

\bigskip
$\begin{array}{cc}
\begin{minipage}[t]{1\textwidth}
{\bf Minhao Hong}\\
College of Arts and Sciences, Shanghai Maritime University, Shanghai 201306, China\\
\texttt{mhhong@shmtu.edu.cn}
\end{minipage}
\hfill
\end{array}$

$\begin{array}{cc}
\begin{minipage}[t]{1\textwidth}
{\bf Qian Yu}\\
Department of Mathematics, Nanjing University of Aeronautics and Astronautics, Nanjing
211106, China\\
Key Laboratory of Mathematical Modelling and High Performance Computing of Air Vehicles (NUAA), MIIT, Nanjing 211106, China\\
\texttt{qyumath@163.com}
\end{minipage}
\hfill
\end{array}$

\end{document}